\documentclass{amsart}

\usepackage{graphicx,enumerate}
\usepackage{amscd}
\usepackage{amsmath, amssymb, amsfonts}
\usepackage{import}
\usepackage{tikz}

\usepackage[
backend=biber,
style=alphabetic,
sorting=ynt,
url = false
]{biblatex}

\newtheorem{theorem}{Theorem}[section]

\newtheorem{corollary}[theorem]{Corollary}
\newtheorem{proposition}[theorem]{Proposition}
\newtheorem{problem}[theorem]{Problem}
\newtheorem{conjecture}[theorem]{Conjecture}

\theoremstyle{definition}
\newtheorem{definition}[theorem]{Definition}

\theoremstyle{remark}

\usepackage{xcolor}

\numberwithin{equation}{section}

\begin{document}

\title{The realization problem of essential surfaces in knot exteriors}

\author{Makoto Ozawa}
\address{Department of Natural Sciences, Faculty of Arts and Sciences, Komazawa University, 1-23-1 Komazawa, Setagaya-ku, Tokyo, 154-8525, Japan}
\email{w3c@komazawa-u.ac.jp}
\thanks{The first author is partially supported by Grant-in-Aid for Scientific Research (C) (No. 17K05262), The Ministry of Education, Culture, Sports, Science and Technology, Japan}


\author{Jes\'{u}s Rodr\'{i}guez-Viorato}
\address{Centro de Investigaci\'{o}n en Matem\'{a}ticas, A.C. 402, 36000 Guanajuato, Gto, M\'{e}xico}
\email{jesusr@cimat.mx }





\begin{abstract}
We study compact orientable essential surfaces in knot exteriors in the 3-sphere. The genus $g$, the number of boundary components $b$, and the boundary slope $p/q$ are fundamental invariants of an essential surface. The \textit{realization problem} asks whether, for a given triple $(g, b, q)$ with $g \ge 0$, $b \ge 1$, and $q \ge 1$, there exists a knot $K \subset S^3$ whose exterior $E(K)$ contains a compact orientable essential surface $F$ of genus $g$ with $b$ boundary components and boundary slope $p/q$ for some $p$.

In general, not all combinations of $(g, b, q)$ are realizable. First, we show that if $b$ is odd, then $q$ must be equal to $1$. Our main theorem states that for any given even $b \ge 2$ and $q \ge 1$, there exist a genus $g \ge 0$ and a knot $K$ such that $E(K)$ contains a compact orientable essential surface with these parameters.
\end{abstract}

\maketitle

\section{Introduction}

Let $K$ be a knot in the $3$-sphere $S^3$. The \textit{knot exterior}, denoted by $E(K)$, is the compact orientable $3$-manifold defined by $E(K) = S^3 - \mathrm{int}\,N(K)$, where $N(K)$ is a regular neighborhood of $K$. A fundamental result of Gordon--Luecke \cite{GL} asserts that knots in $S^3$ are completely determined by their exteriors: two knots are ambient isotopic if and only if their exteriors are homeomorphic.

Another cornerstone in knot theory is Thurston's hyperbolization theorem \cite{T}, which provides a geometric classification of knots into four mutually exclusive types: the trivial knot, torus knots, satellite knots, and hyperbolic knots. This classification is closely linked to the existence of \textit{essential surfaces} in the knot exterior. A compact orientable surface $F$ properly embedded in $E(K)$ is said to be essential if it is incompressible, boundary-incompressible, and not boundary-parallel.

Among the extrinsic features of essential surfaces, the notion of \textit{boundary slope} plays a central role. Let $S$ be a Seifert surface of $K$. The isotopy class of $\partial (S \cap E(K))$ on $\partial E(K)$ is termed a \textit{longitude}, while the isotopy class of an essential loop on $\partial E(K)$ that bounds a disk in $N(K)$ is termed a \textit{meridian}. Upon fixing a meridian-longitude pair $(m, l)$, the isotopy class of any essential loop $\alpha$ on $\partial E(K)$ corresponds to a homology class $[\alpha] = p[m] + q[l] \in H_1(\partial E(K))$. We identify this class with the slope $p/q \in \mathbb{Q} \cup \{\infty\}$, where $\infty$ corresponds to $1/0$.

If $F$ is a compact orientable essential surface with non-empty boundary, each boundary component of $F$ determines the same slope $p/q$, known as the boundary slope of $F$. The set of all boundary slopes for a knot $K$ is denoted by $\mathcal{B}(K)$. For instance, $\mathcal{B}(K)=\{0\}$ for the trivial knot, while $\mathcal{B}(K)=\{0, pq\}$ for the $(p, q)$-torus knot. Hatcher \cite{H} proved that $\mathcal{B}(K)$ is always a finite set.

Essential surfaces are characterized by both intrinsic data, the genus $g(F)$ and the number of boundary components $b = |\partial F|$, and extrinsic data, namely the boundary slope. In this paper, we investigate the interplay between these three quantities by considering the following realization problem:

\begin{problem}[Realization Problem]\label{fundamental}
Given integers $g \ge 0$, $b \ge 1$, and $q \ge 1$, does there exist a knot $K \subset S^3$ such that $E(K)$ contains a compact orientable essential surface $F$ of genus $g$ with $b$ boundary components and boundary slope $p/q$ for some integer $p$? 
\end{problem}

It is well-known that not all triples $(g, b, q)$ are realizable. For $b=1$, homological constraints force the boundary slope to be $0/1$ (i.e., $q=1$), implying that $F$ must be a Seifert surface. Thus, triples $(g, 1, 1)$ are realizable for all $g \ge 0$, while $(g, 1, q)$ with $q \ge 2$ are not. For planar surfaces ($g=0$), Gordon--Luecke \cite{GL87} showed that $q$ must be $1$. Furthermore, if $g=0$ and $b$ is odd, then $b=1$ and $K$ must be the trivial knot. For $g=0$ and $b=2$, $K$ must be a cable knot. Gordon \cite{KO} also proved that the triple $(0, 4, 1)$ is not realizable. These constraints motivate the following conjecture:

\begin{conjecture}[Strong Cabling Conjecture]\label{strong}
For any even integer $b \ge 6$, the triple $(0, b, 1)$ is not realizable.
\end{conjecture}


If one considers arbitrary ambient $3$-manifolds, every triple $(g, b, q)$ is realizable. Thus, the realization problem specifically highlights the topological constraints imposed by the $3$-sphere. Quantitative bounds on the Euler characteristic $\chi(F)$ are known for specific classes: alternating knots satisfy $-\chi(F) \ge \frac{1}{8}bq(n+2)$ \cite{MT}, and Montesinos knots generally satisfy $-\chi(F) \ge bq$ \cite{IM}. These suggest a universal bound of the form $-\chi(F) \ge f(b, q)$.

\subsection{Results}

Let $F$ be a compact orientable essential surface in $E(K)$. We first establish that if $b$ is odd, $F$ must be non-separating in $E(K)$. Since non-separating essential surfaces in knot exteriors necessarily have boundary slope zero ($q=1$), it follows that $q$ must be $1$ whenever $b$ is odd.

While it is known that $(g, b, 1)$ is realizable for some $g$ when $b$ is odd, our main result provides a comprehensive existence theorem for the remaining cases:

\begin{theorem}\label{main}
For any even integer $b \ge 2$ and any integer $q \ge 1$, there exists an integer $g \ge 0$ such that the triple $(g, b, q)$ is realizable.
\end{theorem}

Recent work by Ishikawa, Mattman, and Shimokawa \cite{IMS} showed that the denominator $q$ can be arbitrarily large for alternating knots. However, their results do not prescribe the genus or boundary components. By contrast, this paper elucidates how the parity of $b$ and the denominator $q$ jointly constrain the existence of essential surfaces, providing a clearer picture of the realization problem.

\section{Preliminaries}
We begin with two general results.

\begin{proposition}\label{prop:odd}
If $b$ is odd, $F$ must be non-separating in $E(K)$.
\end{proposition}
\begin{proof}
Notice that all components of $\partial F$ are parallel essential curves on the torus $\partial E(K)$.
Therefore, we can choose a curve $c$ transverse to $\partial F$ that intersects each boundary component exactly once.
After fixing orientations of $c$ and $F$, we assign a sign to each point of intersection in $F \cap c$:
an intersection is positive if $c$ is parallel to the normal vector of $F$ (see Figure~\ref{fig:positive-orientation-c}),
and negative otherwise.

Traversing the curve $c$ and label its successive intersection points with $F$ by
$p_1, p_2, \dots, p_b$.
For each $i = 1, \dots, b$, the subarc $c_i$ of $c$ from $p_i$ to $p_{i+1}$
contains no other intersection points with $F$, where we set $p_{b+1} = p_1$.
If the intersection points $p_i$ and $p_{i+1}$ have the same sign, then the arc $c_i$ connects opposite sides of $F$; in particular, $F$ is non-separating.
Suppose instead that none of the arcs $c_i$ connects opposite sides of $F$.
Then, for each $i$, the points $p_i$ and $p_{i+1}$ must have opposite signs (as illustrated in Figure~\ref{fig:positive-orientation-c}).
As a consequence, the signs at the endpoints of each arc $c_i$ cancel, and we obtain $$0 = \sum_{i=1}^b \bigl(\operatorname{sign}(p_i) + \operatorname{sign}(p_{i+1})\bigr)
  = 2 \sum_{i=1}^b \operatorname{sign}(p_i).
$$ This implies $\sum_{i=1}^b (\pm 1) = 0,$ and hence $b$ must be even, contradicting the assumption that $b$ is odd.

\begin{figure}
    \centering
\begin{tikzpicture}[
    x={(0.9cm,-0.15cm)},  
    y={(0.55cm,0.45cm)},  
    z={(0cm,1cm)},        
    scale=0.7,
    >=stealth,            
    font=\large
]

    \coordinate (E1) at (0,0,0);
    \coordinate (E2) at (6,0,0);
    \coordinate (E3) at (8,4,0);
    \coordinate (E4) at (2,4,0);

    \coordinate (LineStart) at (1, 2, 0);
    \coordinate (LineEnd)   at (7, 2, 0);

    \coordinate (L_Base1) at (2.2, 0.5, 0);
    \coordinate (L_Base2) at (2.8, 3.5, 0);
    \coordinate (L_Top1)  at (2.2, 0.5, 2.5);
    \coordinate (L_Top2)  at (2.8, 3.5, 2.5);
    \coordinate (L_NormalStart) at (2.5, 2.8, 2.0);

    \coordinate (R_Base1) at (5.2, 0.5, 0);
    \coordinate (R_Base2) at (5.8, 3.5, 0);
    \coordinate (R_Top1)  at (5.2, 0.5, 2.5);
    \coordinate (R_Top2)  at (5.8, 3.5, 2.5);
    \coordinate (R_NormalStart) at (5.5, 2.8, 2.0);


    \draw[thick, fill=gray!10] (E1) -- (E2) -- (E3) -- (E4) -- cycle;
    \node[right] at (6.5, 0, 0) {$\partial E$};

    \draw[->, very thick, red!80!black] (LineStart) -- (LineEnd) node[pos=0.97, above, text=red!80!black] {$c$};

    \filldraw[thick, draw=blue!40!black, fill=cyan!20, fill opacity=0.7] 
        (L_Base1) -- (L_Base2) -- (L_Top2) -- (L_Top1) -- cycle;
    
    \draw[->, very thick, blue!70!black] (L_NormalStart) -- ++(1.2, 0, 0);
    
    \node[anchor=north west, black] at (2.5, 2, 0) {$+$};

    \filldraw[thick, draw=blue!40!black, fill=cyan!20, fill opacity=0.7] 
        (R_Base1) -- (R_Base2) -- (R_Top2) -- (R_Top1) -- cycle;

    \node[right, blue!40!black] at (R_Top2) {$F$};

    \draw[->, very thick, blue!70!black] (R_NormalStart) -- ++(-1.2, 0, 0);

    \node[anchor=north west, black] at (5.5, 2, 0) {$-$};

\end{tikzpicture}
    \caption{Positive intersection }
    \label{fig:positive-orientation-c}
\end{figure}
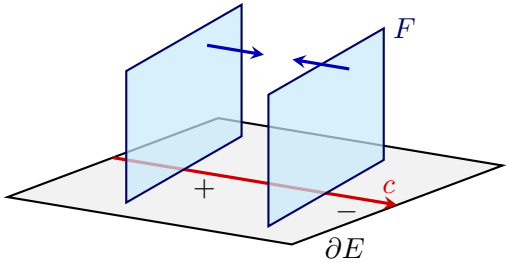

\end{proof}
\begin{proposition}\label{prop:non-separating}
If $F$ is orientable and non-separating in $E(K)$, then the boundary slope of $F$ is zero; that is, $p = 0$ and $q = 1$.
\end{proposition}
\begin{proof}
Observe that $F$ must have a non-empty boundary. Otherwise, $F$ would be separating in $E(K)$.

Fix an orientation of $F$, which determines a homology class $w = [F, \partial F] \in H_1(E(K), \partial E(K)) \cong \mathbb{Z}$.
Because $F$ is non-separating, there exists a curve $c$ in $E(K)$ that intersects $F$ only once. We can find $c$ by taking a curve connecting opposite sides of $F$ and sliding its ends near $F$ to coincide in $F$, forming the loop $c$.
After fixing an orientation on $c$, it represents an element $\alpha = [c]$ in  $H_1(E(K))\cong \mathbb{Z}$.
Take the cap product $H^1(E(K)) \times H_1(E(K)) \to H^0(E(K)) \cong \mathbb{Z}$.
By Poincar\'{e} duality we identify $H_2(E(K), \partial E(K))$ with $H^1(E(K))$, and hence obtain the intersection pairing. This defines a bilinear form $\beta: H_2(E(K), \partial E(K)) \times H_1(E(K)) \to \mathbb{Z}$ that takes any oriented surface $(S, \partial S)$ in $(E(K), \partial E(K))$ and an oriented loop $\gamma$ in $E(K)$ and assign $\beta([S,\partial S], [\gamma])$ as the sum of intersections with signs between $S$ and $\gamma$.
This is described via transversality.

As $\beta([F, \partial F], [c]) = 1$, both $[F, \partial F]$ and $[c]$ must be primitive in their respective homology group.
As $H_2(E(K), \partial E(K)) \cong  \mathbb{Z}$ is generated by a Seifert surface $[S_0, \partial S_0]$, it must be that $[F, \partial F] = \pm  [S_0, \partial S_0]$ in $H_2(E(K), \partial E(K))$.
Now consider the connecting homomorphism $\delta: H_2(E(K), \partial E(K)) \to H_1(\partial E(K))$.
The group $H_1(\partial E(K))$  is generated by a meridian $\mu$ (a closed curve that bounds a disk $D$ in $S^3$ and intersects $K$ in a single point) and a longitude $\lambda$ ( $= \partial S_0$).
Now $[\partial F] = \delta([F, \partial F]) = \pm \delta([S_0, \partial S_0])= \pm [ \partial S_0] = \pm [\lambda] = 0 [\mu]+ (\pm1)[\lambda]$.
This implies that the slope of $\partial F$ is zero in $\partial E(K)$.
\end{proof}


We restrict to even values of $b$ in Theorem \ref{main} because the parity of $b$ affects the orientability and separability properties of the constructed surfaces, as shown in Propositions \ref{prop:odd} and \ref{prop:non-separating}.
These constraints ensure the consistency of boundary slope realizability.

\section{Proof}
Theorem \ref{main} will result as a consequence of Theorems \ref{thm:2even-boundaries} and \ref{thm:2odd-boundaries-simplier}.
This is proved using the techniques from \cite{HO}. A brief introduction to these techniques can be found in \cite{AR}.
On all the proofs presented here, we will give a system of paths (one per tangle) that travels on the edges of Diagram $\mathcal{D}$ depicted on Figure \ref{fig:diagram_D} (see \cite{HO} for the definition).
\begin{figure}
    \centering
    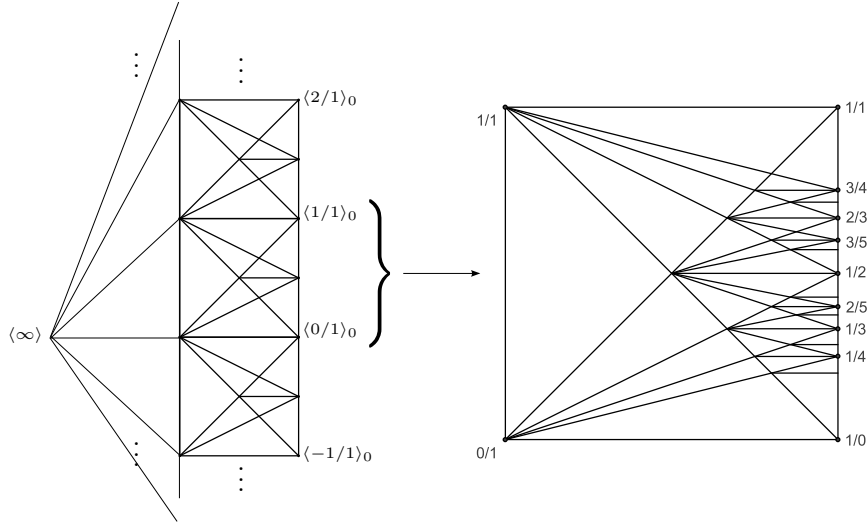
    \caption{Diagram $\mathcal{D}$}
    \label{fig:diagram_D}
\end{figure}

The conditions for a system of paths to describe a surface are given in the following definition from \cite{HO}.
\begin{definition}
An \emph{edgepath system} for a Montesinos knot $K$ is a set of $n$ edgepaths $\gamma_{1}$, $\gamma_2$, \dots $\gamma_n$ on the 1-skeleton of $\mathcal{D}$ satisfying the following conditions:

\begin{itemize}\label{def:condiciones-de-sistema-de-caminos}
\item[(E1)] The starting point of $\gamma_i$ lies on the edge $\langle
p_i/q_i ,  p_i/q_i \rangle$, and if the starting point is not the vertex
$\langle p_i/q_i \rangle$, then the edgepath $\gamma_i$ is constant.
\item[(E2)] $\gamma_i$ is \emph{minimal}, i.e., it never stops or retraces
itself, nor does it ever go along two sides of the same triangle of
$\mathcal{D}$ in succession.
\item[(E3)] The ending points of the $\gamma_i$'s are rational points of
$\mathcal{D}$ which all lie on one vertical line and whose vertical coordinates
add up to zero.
\item[(E4)] $\gamma_i$ proceeds monotonically from right to left,
``monotonically'' in the weak sense that motion along vertical edges is
permitted.
\end{itemize}
\end{definition}

 \begin{theorem}\label{thm:surface_on_-ppq}
 The pretzel knot $K = P(-p, p, q)$, with $p \geq 3$ and $q \geq  a^2 - a +1$ odd integers, has an incompressible surface with $gcd(a+b, 2a^2)/gcd(a,b)$ boundary components and a slope denominator $(a+b)/gcd(a+b, 2a^2)$ where $a=p-1$, $b=q-1$.
 
 \emph{Note}: here $gcd(\cdot,\cdot)$ is the greatest common divisor function.
 \end{theorem}
  
\begin{proof}
We take the path systems given by 
\begin{align*}
    \gamma_0&:& \langle \frac{-1}{p}\rangle &\longrightarrow \frac{x_0}{x_0+y_0} \langle \frac{-1}{p}\rangle + \frac{y_0}{x_0+y_0}\langle \frac{-1}{p-1}\rangle \\
    \gamma_1&:& \langle \frac{1}{p}\rangle &\longrightarrow \frac{x_1}{x_1+y_1} \langle \frac{1}{p}\rangle + \frac{y_1}{x_1+y_1} \langle \frac{0}{1}\rangle \\
    \gamma_2&: &\langle \frac{1}{q}\rangle &\longrightarrow \frac{x_2}{x_2+y_2} \langle \frac{1}{q}\rangle + \frac{y_2}{x_2+y_2}  \langle \frac{0}{1}\rangle \\
\end{align*}

We write ending points of the edgepaths in $(a,b,c)$ coordinates, 

\begin{align*}
     \text{End}(\gamma_0)&:& (x_0+y_0, x_0(p-1)+y_0(p-2), -x_0-y_0) \\
    \text{End}(\gamma_1)&:& (x_1+y_1, x_1(p-1), x_1) \\
    \text{End}(\gamma_2)&:& (x_2+y_2, x_2(q-1), x_2) \\
\end{align*}

Since each $a$-coordinate equals the number of sheets, they must be all 
equal.
By condition E3, all $b$ coordinates need to be equal, and the sum of all $c$ coordinates must be zero.
These observations translate into the following equations:
\begin{align*}
    x_0 + y_0 &= x_1 + y_1 = x_2+ y_2\\
    x_0(p-1) + y_0(p-2) &= x_1(p-1) =x_2(q-1)\\
    x_1+x_2 &= x_0+y_0
\end{align*}

By elementary number-theoretic arguments, one verifies that all solutions are multiples of the solution:
      \begin{align*}
          x_0= \frac{a+b-a^2}{g} &, &y_0 = \frac{a^2}{g} &, &x_1 =\frac{b}{g} &, &y_1 = \frac{a}{g} &, & x_2 =\frac{a}{g} &, &y_2 = \frac{b}{g}\\
      \end{align*}

where $g 
= gcd(a,b)$ (recall that $a =p-1$ and $b=q-1$). Here, we
need the condition $q \geq a^2 -a + 1$ to make $x_0$ non-negative.
So, the three edgepaths have a fractional longitude given by:
\begin{align*}
|\gamma_0| = \frac{y_0}{x_0+y_0} = \frac{a^2}{a+b} &, & |\gamma_1| = \frac{a}{a+b}&,& |\gamma_2|
= \frac{b}{a+b}
\end{align*}

Then the number of sheets has to be an integer $s$ such that $s|\gamma_i|$ is an integer for $i=0, 1,2$.
This implies that $s$ is a multiple of $(a+b)/gcd(a,b)$. This means that to construct a connected surface, it is enough to consider when the \emph{number of sheets} is  $s = (a+b)/gcd(a,b)$.
Now, the twist $\tau(\gamma_i)$, in this case, is twice the length of each edgepath with a sign (negative if it goes downwards and positive if it goes upwards in $\mathcal{D}$):

\begin{align*}
\tau(\gamma_0) =  \frac{2a^2}{a+b} &, & \tau(\gamma_1) = -\frac{2a}{a+b}&,& \tau(\gamma_2) = -\frac{2b}{a+b}
\end{align*}

The total twist is: $$\tau(\gamma_0, \gamma_1, \gamma_2) = \frac{2a^2-2a -2b}{a+b}$$

And by  \cite{HO} we know that the actual slope of the surface $slope(F)$ satisfy that $\displaystyle slope(F) \equiv \tau(\gamma_0, \gamma_1, \gamma_2) \equiv \frac{2a^2}{a+b}\pmod{1}$.
This means that the denominator of $slope(F)$, as an irreducible fraction,  is equal to $(a+b)/gcd(a+b, 2a^2)$.
Finally, the number of boundary components must be: $$|\partial F|=sheets(F)/den(\partial F) = \frac{(a+b)/gcd(a,b)}{(a+b)/gcd(a+b, 2a^2)} = \frac{gcd(a+b, 2a^2)}{gcd(a,b)}.$$

This determines the slope of the candidate.
It remains to be proven that it is incompressible and orientable.
For incompressibility, we use {\cite[Corollary 2.5]{HO}}, which we state below.
\begin{corollary}[Corollary 2.5 from \cite{HO}] \label{cor:HO-2.5}
A candidate surface is incompressible unless the cycle of $r$-values for the final edges of the $\gamma_i$'s is of one the following types: $(0, r_2, \dots, r_n)$, $ (1, \dots , 1, r_n)$, or $(1, \dots , 1, 2, r_n)$.
\end{corollary}

The mentioned cycle of $r$-values can be computed from the last edge of each edgepath $\gamma_i$ as follows.
If $\langle p/q, u/v\rangle$ is the final edge of $\gamma_i$ (going from $p/q$ towards $u/v$),  $r_i =  q - v$ is the $r$-value corresponding to $\gamma_i$.
Then,  $(r_1, r_2, \dots, r_n) $ is the cycle of $r$-values for the final edges.
In our case, the final edges (the only edge per edgepath) are $\langle1/p, 1/p-1\rangle$, $\langle 1/p, 0/1 \rangle$ and $\langle 1/q, 0/1 \rangle$.
So the cycle of $r$-values is  $(1, p-1, q-1)$, which is none of the exceptional cases from Corollary \ref{cor:HO-2.5}.
So, our candidate surface is incompressible.

The orientability of the candidate surface is discussed in Theorem \ref{thm:orientability_of_surfeces_on_-ppq}.
\end{proof}

As for orientability, the surfaces from Theorem \ref{thm:surface_on_-ppq} are orientable only in the following cases:

\begin{theorem}\label{thm:orientability_of_surfeces_on_-ppq}
The surface given by Theorem \ref{thm:surface_on_-ppq} is orientable if $g=\gcd(a,b)$ is even and $(a+b)/g$ is even.
\end{theorem}
\begin{proof}

Candidate surfaces are constructed by gluing embedded surfaces $S_i$ inside each tangle $p_i/q_i$;
we denote by $B_i$ the ball containing the tangle.
Here, $p_0/q_0 = -1/p$, $p_1/q_1 = 1/p$, and $p_2/q_2 = 1/q$. The edgepath $\gamma_i$ encodes the surfaces $S_i \subset B_i$, which are generally disconnected.
In our case, the edgepaths $\gamma_0$, $\gamma_1$, and $\gamma_2$ are described in the proof of Theorem~\ref{thm:surface_on_-ppq}.

One can check that the surfaces $S_0$, $S_1$, and $S_2$ are all disconnected.
Our strategy to prove orientability is to orient each component of $S_i$ so that they induce opposite orientations on $B_i \cap B_j$.
This ensures that the chosen orientations extend to the entire surface, thereby proving orientability.

We need first to recall how each $S_i$ is constructed. We start by considering a collar neighborhood $C = S^2 \times I $ of the $\partial B_i$ inside $B_i$ where $S\times \{1\} = \partial B_i$.
Put the first part of the arcs of the tangle parallel to $I$ following the product structure of $C$, and the tangled part of the arcs resting over the sphere $S^2 \times \{0\}$ (as in Figure \ref{fig:parallel_copies}).
Now take $s$ parallel copies of slope $p_i/q_i$ over $S^2 \times \{0\}$;
over the neighborhood of the $p_i/q_i$ arcs part of the Montesinos's decomposition of the original knot $K$ (see Figure \ref{fig:parallel_copies}).
Product these parallel copies with the $[0, \epsilon]$ inside $C$, creating in this way a surface $S'_i$ with $2s$ disk components contained in $B_i$.
\begin{figure}
        \centering
        \includegraphics[width=0.5\linewidth]{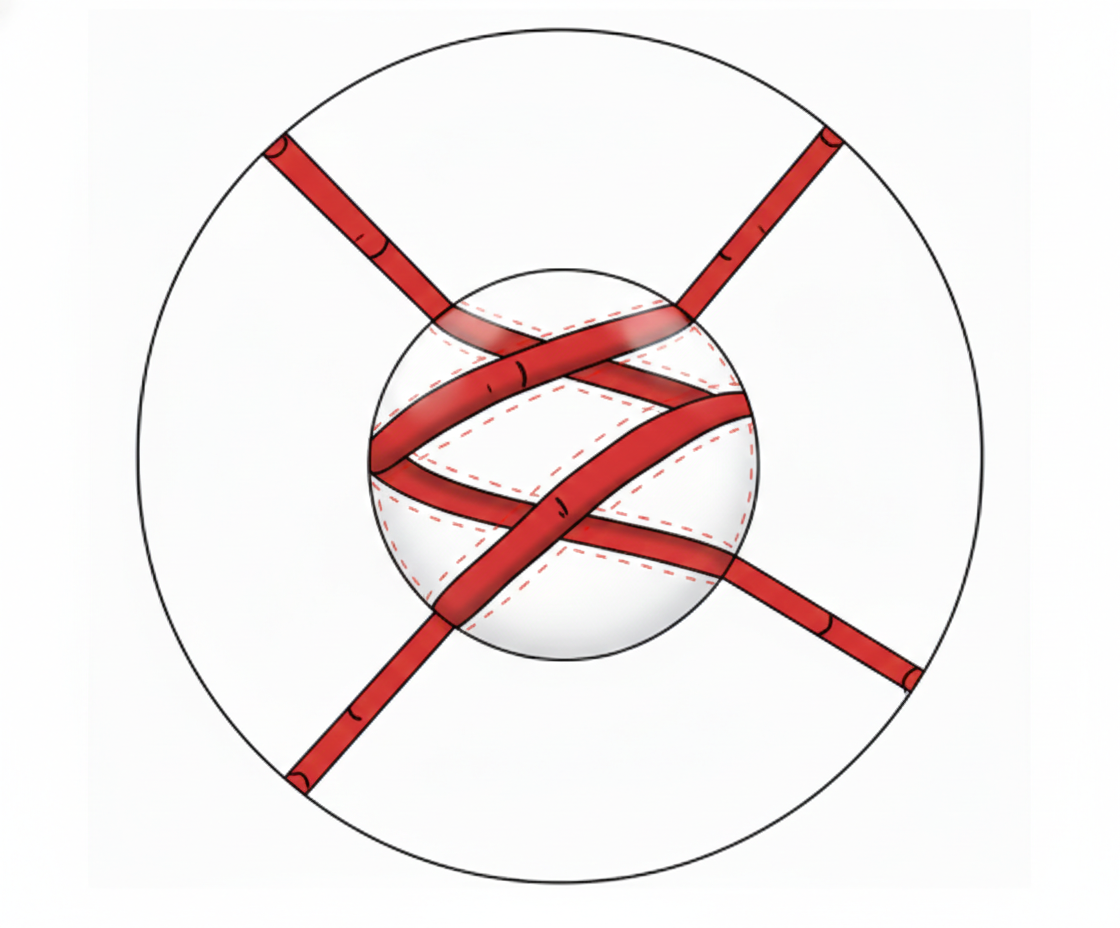}
        \caption{Parallel copies of the slope $1/2$ over the regular neighborhood of the tangle $1/2$.}
        \label{fig:parallel_copies}
    \end{figure}


To construct $S_i$ from $S'_i$, we attach saddle bands according to the edgepath $\gamma_i$ and transverse to the product structure of the collar $C = S^2 \times I$, starting from level $S^2 \times \epsilon$ and moving upwards toward $S^2 \times 1  = \partial B_i$.

For the case that concern us,  the edgepath $\gamma_i$ prescribe a total of $y_i$ saddle of the same type for   $i=0,1,2$ . For example, for $\gamma_1$, we need to attach $y_1 = \tfrac{a}{g}$ saddles, each transforming a pair of arcs of slope $1/p$ into a pair of arcs of slope $0$. So, we start at level $S^2 \times \epsilon$ with $s = a/g + b/g $ parallel copies of a pair of arcs of slope $1/p$ and we end up at level $S^2 \times 1$ with $a/g$ of slope $1/p$ plus $b/g$ of slope $0$ . This is achieved using the saddle move illustrated in Figure~\ref{fig:saddle10}.

Now, to fix an orientation for $S_i$ we will first fix one for $S'_i$; which has $2^{2s}$ possibilities (two per component of $S'_i$). We can picture these orientations by their induced orientation in $\partial S'_i$ on $S^2 \times \epsilon$; which are $s$-parallel copies of directed pairs of arcs of slope $p_i/q_i$. But not every orientation can be extended to an orientation of $S_i$. We need to make sure that the saddle preserve that orientation.

For a pair of arcs $\alpha$ and $\beta$ of slope $p_i/q_i$, the two arcs must be oriented so that the saddle being use connects opposite sides of $\alpha$ and $\beta$. For the specific case of saddle changing the slope from $1/p$ to $0$, the saddle preserve orientation if the arcs $\alpha$ and $\beta$ are directed as shown in Figure~\ref{fig:saddle10}, or by reversing both simultaneously. Fix any of thoses orientation on $\alpha$  and $\beta$.

\begin{figure}
    \centering
\begingroup%
  \makeatletter%
  \providecommand\color[2][]{%
    \errmessage{(Inkscape) Color is used for the text in Inkscape, but the package 'color.sty' is not loaded}%
    \renewcommand\color[2][]{}%
  }%
  \providecommand\transparent[1]{%
    \errmessage{(Inkscape) Transparency is used (non-zero) for the text in Inkscape, but the package 'transparent.sty' is not loaded}%
    \renewcommand\transparent[1]{}%
  }%
  \providecommand\rotatebox[2]{#2}%
  \newcommand*\fsize{\dimexpr\f@size pt\relax}%
  \newcommand*\lineheight[1]{\fontsize{\fsize}{#1\fsize}\selectfont}%
  \ifx\svgwidth\undefined%
    \setlength{\unitlength}{187.50000721bp}%
    \ifx\svgscale\undefined%
      \relax%
    \else%
      \setlength{\unitlength}{\unitlength * \real{\svgscale}}%
    \fi%
  \else%
    \setlength{\unitlength}{\svgwidth}%
  \fi%
  \global\let\svgwidth\undefined%
  \global\let\svgscale\undefined%
  \makeatother%
  \begin{picture}(1,0.86264772)%
    \lineheight{1}%
    \setlength\tabcolsep{0pt}%
    \put(0,0){\includegraphics[width=\unitlength,page=1]{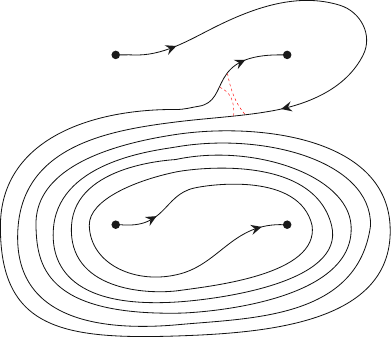}}%
    \put(0.34328438,0.75413099){\makebox(0,0)[lt]{\lineheight{1.25}\smash{\begin{tabular}[t]{l}$\alpha$\end{tabular}}}}%
    \put(0.3004215,0.59641369){\makebox(0,0)[lt]{\lineheight{1.25}\smash{\begin{tabular}[t]{l}$\beta$\end{tabular}}}}%
  \end{picture}%
\endgroup%

    \caption{Example of a saddle move from slope $1/7$ to slope $0$.}
    \label{fig:saddle10}
\end{figure}


Next, consider the $s$ parallel copies $\alpha_1,\ldots,\alpha_s$ of $\alpha$ and $\beta_1,\ldots,\beta_s$ of $\beta$ in $\partial S'_i \subset S \times {\epsilon}$, numbered from leftmost to rightmost (see Figure~\ref{fig:parallelcopies}). Each arc $\alpha_j$ may be oriented either consistently with $\alpha$ or with the opposite orientation. We encode this choice by a vector

\[
\mathbf{x}=(x_1,\ldots,x_s)\in\{\pm1\}^s,
\]

and similarly we encode the orientations of the arcs $\beta_i$ by

\[
\mathbf{y}=(y_1,\ldots,y_s)\in\{\pm1\}^s.
\]
\begin{figure}
    \centering
\begingroup%
  \makeatletter%
  \providecommand\color[2][]{%
    \errmessage{(Inkscape) Color is used for the text in Inkscape, but the package 'color.sty' is not loaded}%
    \renewcommand\color[2][]{}%
  }%
  \providecommand\transparent[1]{%
    \errmessage{(Inkscape) Transparency is used (non-zero) for the text in Inkscape, but the package 'transparent.sty' is not loaded}%
    \renewcommand\transparent[1]{}%
  }%
  \providecommand\rotatebox[2]{#2}%
  \newcommand*\fsize{\dimexpr\f@size pt\relax}%
  \newcommand*\lineheight[1]{\fontsize{\fsize}{#1\fsize}\selectfont}%
  \ifx\svgwidth\undefined%
    \setlength{\unitlength}{179.91075771bp}%
    \ifx\svgscale\undefined%
      \relax%
    \else%
      \setlength{\unitlength}{\unitlength * \real{\svgscale}}%
    \fi%
  \else%
    \setlength{\unitlength}{\svgwidth}%
  \fi%
  \global\let\svgwidth\undefined%
  \global\let\svgscale\undefined%
  \makeatother%
  \begin{picture}(1,0.22477119)%
    \lineheight{1}%
    \setlength\tabcolsep{0pt}%
    \put(0,0){\includegraphics[width=\unitlength,page=1]{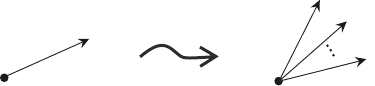}}%
    \put(0.12780067,0.10142112){\makebox(0,0)[lt]{\lineheight{1.25}\smash{\begin{tabular}[t]{l}$\alpha$\end{tabular}}}}%
    \put(0.75145946,0.18153548){\makebox(0,0)[lt]{\lineheight{1.25}\smash{\begin{tabular}[t]{l}$\alpha_1$\end{tabular}}}}%
    \put(0.8247738,0.13769795){\makebox(0,0)[lt]{\lineheight{1.25}\smash{\begin{tabular}[t]{l}$\alpha_2$\end{tabular}}}}%
    \put(0.88401725,0.05740646){\makebox(0,0)[lt]{\lineheight{1.25}\smash{\begin{tabular}[t]{l}$\alpha_s$\end{tabular}}}}%
  \end{picture}%
\endgroup%

    \caption{Taking parallel copies of $\alpha$ and numbering from leftmost to rigthmost}
    \label{fig:parallelcopies}
\end{figure}
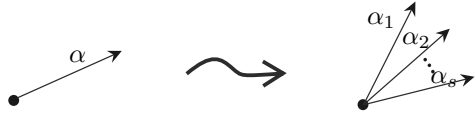

Orientability is preserved after attaching the saddles provided that the orientation choices on the parallel arcs $\alpha_j$ and $\beta_j$ agree both with the one of $\alpha$ and $\beta$ or its reverse, that is, provided that $\mathbf{x}=\mathbf{y}$. A convenient choice satisfying this condition is to alternate signs, for instance
$$\mathbf{x}=\mathbf{y}=(+1,-1,+1,\ldots,+1,-1).$$

After performing $\tfrac{a}{g}$ saddle moves, we obtain an induced orientation on $\partial B$, shown schematically in Figure~\ref{fig:orientationafter}. In this configuration, the orientation data are given by
\[
\mathbf{x_1}=(-1,+1,-1,+1,\ldots,-1), \qquad
\mathbf{y_1}=(+1,-1,+1,-1,\ldots,+1),
\]
where \(|\mathbf{x_1}|=\tfrac{a}{g}\) and \(|\mathbf{y_1}|=\tfrac{b}{g}\). Here \(|\cdot|\) denotes the length of the vector. 

Notice that $\mathbf{x_1}$ starts in $-1$ and ends in $-1$ because $a/g$ is odd, and the same applies for $b/g$.

\begin{figure}[h]
  \centering
\begingroup%
  \makeatletter%
  \providecommand\color[2][]{%
    \errmessage{(Inkscape) Color is used for the text in Inkscape, but the package 'color.sty' is not loaded}%
    \renewcommand\color[2][]{}%
  }%
  \providecommand\transparent[1]{%
    \errmessage{(Inkscape) Transparency is used (non-zero) for the text in Inkscape, but the package 'transparent.sty' is not loaded}%
    \renewcommand\transparent[1]{}%
  }%
  \providecommand\rotatebox[2]{#2}%
  \newcommand*\fsize{\dimexpr\f@size pt\relax}%
  \newcommand*\lineheight[1]{\fontsize{\fsize}{#1\fsize}\selectfont}%
  \ifx\svgwidth\undefined%
    \setlength{\unitlength}{187.49998558bp}%
    \ifx\svgscale\undefined%
      \relax%
    \else%
      \setlength{\unitlength}{\unitlength * \real{\svgscale}}%
    \fi%
  \else%
    \setlength{\unitlength}{\svgwidth}%
  \fi%
  \global\let\svgwidth\undefined%
  \global\let\svgscale\undefined%
  \makeatother%
  \begin{picture}(1,1.1399205)%
    \lineheight{1}%
    \setlength\tabcolsep{0pt}%
    \put(0,0){\includegraphics[width=\unitlength,page=1]{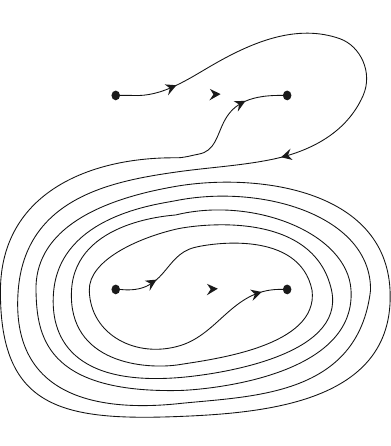}}%
    \put(0.43899282,0.96396473){\makebox(0,0)[lt]{\lineheight{1.25}\smash{\begin{tabular}[t]{l}$\mathbf{x_1}$\end{tabular}}}}%
    \put(0.37998228,0.84323087){\makebox(0,0)[lt]{\lineheight{1.25}\smash{\begin{tabular}[t]{l}$\mathbf{y_1}$\end{tabular}}}}%
    \put(0,0){\includegraphics[width=\unitlength,page=2]{orientationafter.pdf}}%
    \put(0.48242796,0.41121808){\makebox(0,0)[lt]{\lineheight{1.25}\smash{\begin{tabular}[t]{l}$\mathbf{y_1}$\end{tabular}}}}%
    \put(0.34784378,0.44679322){\makebox(0,0)[lt]{\lineheight{1.25}\smash{\begin{tabular}[t]{l}$\mathbf{x_1}$\end{tabular}}}}%
    \put(0,0){\includegraphics[width=\unitlength,page=3]{orientationafter.pdf}}%
    \put(0.5132991,1.09771598){\color[rgb]{1,0,0}\makebox(0,0)[lt]{\lineheight{1.25}\smash{\begin{tabular}[t]{l}$z$-axis\end{tabular}}}}%
  \end{picture}%
\endgroup%

  \caption{Orientation after $\frac{a}{g}$ saddle moves.}
  \label{fig:orientationafter}
\end{figure}

Cutting the diagram along the $z$--axis (indicated by the dotted red line in Figure~\ref{fig:orientationafter}) produces two disks, or equivalently two half-planes, which we denote by $D_1^L$ and $D_1^R$ (see Figure~\ref{fig:halves1}).

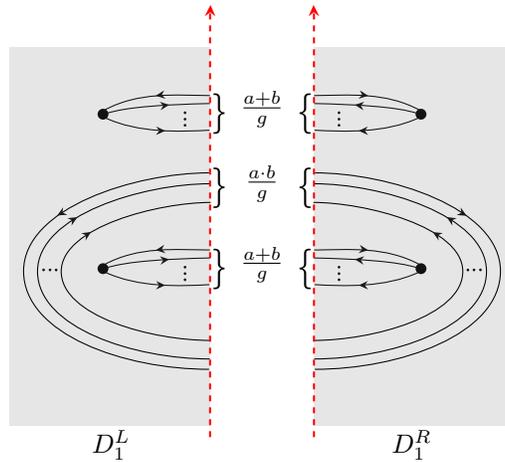
\begin{figure}[h]
  \centering
\begingroup%
  \makeatletter%
  \providecommand\color[2][]{%
    \errmessage{(Inkscape) Color is used for the text in Inkscape, but the package 'color.sty' is not loaded}%
    \renewcommand\color[2][]{}%
  }%
  \providecommand\transparent[1]{%
    \errmessage{(Inkscape) Transparency is used (non-zero) for the text in Inkscape, but the package 'transparent.sty' is not loaded}%
    \renewcommand\transparent[1]{}%
  }%
  \providecommand\rotatebox[2]{#2}%
  \newcommand*\fsize{\dimexpr\f@size pt\relax}%
  \newcommand*\lineheight[1]{\fontsize{\fsize}{#1\fsize}\selectfont}%
  \ifx\svgwidth\undefined%
    \setlength{\unitlength}{190.03343717bp}%
    \ifx\svgscale\undefined%
      \relax%
    \else%
      \setlength{\unitlength}{\unitlength * \real{\svgscale}}%
    \fi%
  \else%
    \setlength{\unitlength}{\svgwidth}%
  \fi%
  \global\let\svgwidth\undefined%
  \global\let\svgscale\undefined%
  \makeatother%
  \begin{picture}(1,0.90234814)%
    \lineheight{1}%
    \setlength\tabcolsep{0pt}%
    \put(0,0){\includegraphics[width=\unitlength,page=1]{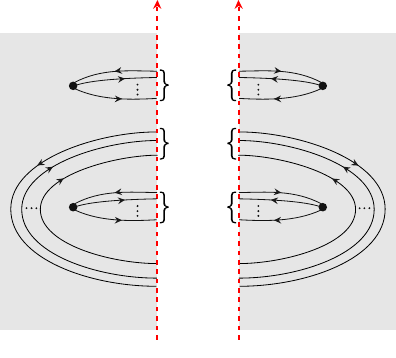}}%
    \put(0.499995,0.68489877){\color[rgb]{0,0,0}\makebox(0,0)[t]{\lineheight{1.25}\smash{\begin{tabular}[t]{c}$\frac{a+b}{g}$\end{tabular}}}}%
    \put(0.20109456,0.01352212){\color[rgb]{0,0,0}\makebox(0,0)[t]{\lineheight{1.25}\smash{\begin{tabular}[t]{c}$D^L_1$\end{tabular}}}}%
    \put(0.7972223,0.01352212){\color[rgb]{0,0,0}\makebox(0,0)[t]{\lineheight{1.25}\smash{\begin{tabular}[t]{c}$D^R_1$\end{tabular}}}}%
    \put(0.499995,0.3766118){\color[rgb]{0,0,0}\makebox(0,0)[t]{\lineheight{1.25}\smash{\begin{tabular}[t]{c}$\frac{a+b}{g}$\end{tabular}}}}%
    \put(0.499995,0.53766754){\color[rgb]{0,0,0}\makebox(0,0)[t]{\lineheight{1.25}\smash{\begin{tabular}[t]{c}$\frac{a\cdot b}{g}$\end{tabular}}}}%
  \end{picture}%
\endgroup%

  \caption{Left and right halves $D_1^L$ and $D_1^R$.}
  \label{fig:halves1}
\end{figure}

An analogous analysis applies to $\gamma_2$. To obtain compatible orientations, we reverse all signs, namely
\[
x_2=(-1,+1,-1,+1,\ldots,-1), \qquad
y_2=(+1,-1,+1,-1,\ldots,+1),
\]
yielding the configuration shown in Figure~\ref{fig:halves2}. At this stage we have constructed two oriented surfaces $F_1$ and $F_2$ which match perfectly along
\[
F_1 \cap F_2 \subset D_1^R = D_2^L = B_1 \cap B_2.
\]

\begin{figure}[h]
  \centering
\begingroup%
  \makeatletter%
  \providecommand\color[2][]{%
    \errmessage{(Inkscape) Color is used for the text in Inkscape, but the package 'color.sty' is not loaded}%
    \renewcommand\color[2][]{}%
  }%
  \providecommand\transparent[1]{%
    \errmessage{(Inkscape) Transparency is used (non-zero) for the text in Inkscape, but the package 'transparent.sty' is not loaded}%
    \renewcommand\transparent[1]{}%
  }%
  \providecommand\rotatebox[2]{#2}%
  \newcommand*\fsize{\dimexpr\f@size pt\relax}%
  \newcommand*\lineheight[1]{\fontsize{\fsize}{#1\fsize}\selectfont}%
  \ifx\svgwidth\undefined%
    \setlength{\unitlength}{190.03343717bp}%
    \ifx\svgscale\undefined%
      \relax%
    \else%
      \setlength{\unitlength}{\unitlength * \real{\svgscale}}%
    \fi%
  \else%
    \setlength{\unitlength}{\svgwidth}%
  \fi%
  \global\let\svgwidth\undefined%
  \global\let\svgscale\undefined%
  \makeatother%
  \begin{picture}(1,0.90234814)%
    \lineheight{1}%
    \setlength\tabcolsep{0pt}%
    \put(0,0){\includegraphics[width=\unitlength,page=1]{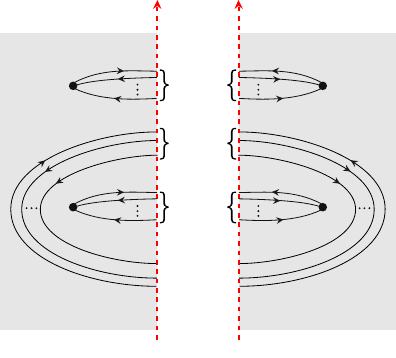}}%
    \put(0.499995,0.68489877){\color[rgb]{0,0,0}\makebox(0,0)[t]{\lineheight{1.25}\smash{\begin{tabular}[t]{c}$\frac{a+b}{g}$\end{tabular}}}}%
    \put(0.20109456,0.01352212){\color[rgb]{0,0,0}\makebox(0,0)[t]{\lineheight{1.25}\smash{\begin{tabular}[t]{c}$D^L_2$\end{tabular}}}}%
    \put(0.7972223,0.01352212){\color[rgb]{0,0,0}\makebox(0,0)[t]{\lineheight{1.25}\smash{\begin{tabular}[t]{c}$D^R_2$\end{tabular}}}}%
    \put(0.499995,0.3766118){\color[rgb]{0,0,0}\makebox(0,0)[t]{\lineheight{1.25}\smash{\begin{tabular}[t]{c}$\frac{a+b}{g}$\end{tabular}}}}%
    \put(0.499995,0.53766754){\color[rgb]{0,0,0}\makebox(0,0)[t]{\lineheight{1.25}\smash{\begin{tabular}[t]{c}$\frac{a\cdot b}{g}$\end{tabular}}}}%
  \end{picture}%
\endgroup%

  \caption{Compatible orientations for $\gamma_2$. Left and right halves $D^L_2$ and $D^R_2$.}
  \label{fig:halves2}
\end{figure}
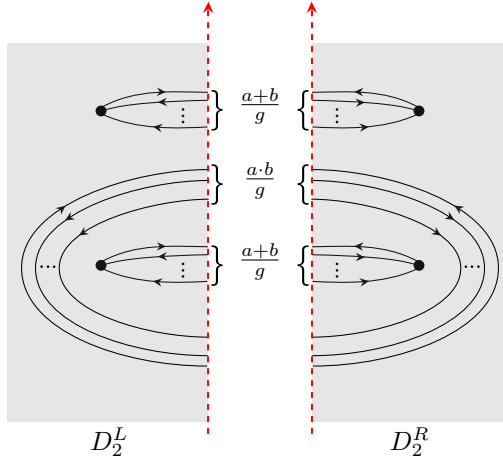

Finally, we apply the same analysis to \(\gamma_0\), using the saddle move from slope \(-1/p\) to \(-1/(p-1)\) shown in Figure~\ref{fig:saddle0}. Choosing alternating orientations for the parallel copies of \(\alpha_0\) and \(\beta_0\) ensures that orientability is preserved after \(\tfrac{a^2}{g}\) saddle moves, yielding the induced orientations on each half of \(\partial B_0\) depicted in Figure~\ref{fig:final}.
.

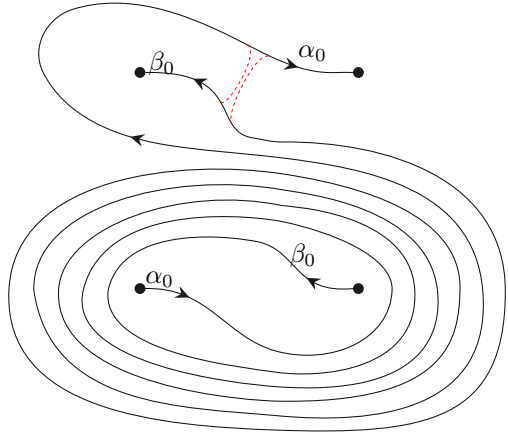
\begin{figure}[h]
  \centering
\begingroup%
  \makeatletter%
  \providecommand\color[2][]{%
    \errmessage{(Inkscape) Color is used for the text in Inkscape, but the package 'color.sty' is not loaded}%
    \renewcommand\color[2][]{}%
  }%
  \providecommand\transparent[1]{%
    \errmessage{(Inkscape) Transparency is used (non-zero) for the text in Inkscape, but the package 'transparent.sty' is not loaded}%
    \renewcommand\transparent[1]{}%
  }%
  \providecommand\rotatebox[2]{#2}%
  \newcommand*\fsize{\dimexpr\f@size pt\relax}%
  \newcommand*\lineheight[1]{\fontsize{\fsize}{#1\fsize}\selectfont}%
  \ifx\svgwidth\undefined%
    \setlength{\unitlength}{187.50000721bp}%
    \ifx\svgscale\undefined%
      \relax%
    \else%
      \setlength{\unitlength}{\unitlength * \real{\svgscale}}%
    \fi%
  \else%
    \setlength{\unitlength}{\svgwidth}%
  \fi%
  \global\let\svgwidth\undefined%
  \global\let\svgscale\undefined%
  \makeatother%
  \begin{picture}(1,0.86264772)%
    \lineheight{1}%
    \setlength\tabcolsep{0pt}%
    \put(0,0){\includegraphics[width=\unitlength,page=1]{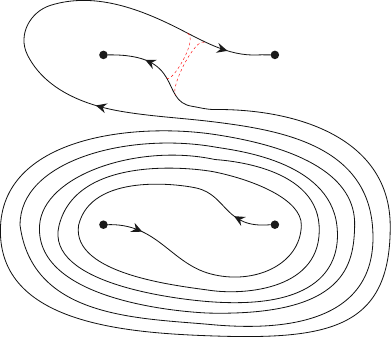}}%
    \put(0.58259751,0.75271562){\makebox(0,0)[lt]{\lineheight{1.25}\smash{\begin{tabular}[t]{l}$\alpha_0$\end{tabular}}}}%
    \put(0.56573935,0.34177472){\makebox(0,0)[lt]{\lineheight{1.25}\smash{\begin{tabular}[t]{l}$\beta_0$\end{tabular}}}}%
    \put(0.28176648,0.72745618){\makebox(0,0)[lt]{\lineheight{1.25}\smash{\begin{tabular}[t]{l}$\beta_0$\end{tabular}}}}%
    \put(0.27574442,0.29833645){\makebox(0,0)[lt]{\lineheight{1.25}\smash{\begin{tabular}[t]{l}$\alpha_0$\end{tabular}}}}%
  \end{picture}%
\endgroup%

  \caption{Saddle move corresponding to the edgepath $\gamma_0$, which changes the slope from $-\tfrac{1}{p}$ to $-\tfrac{1}{p-1}$.}
  \label{fig:saddle0}
\end{figure}

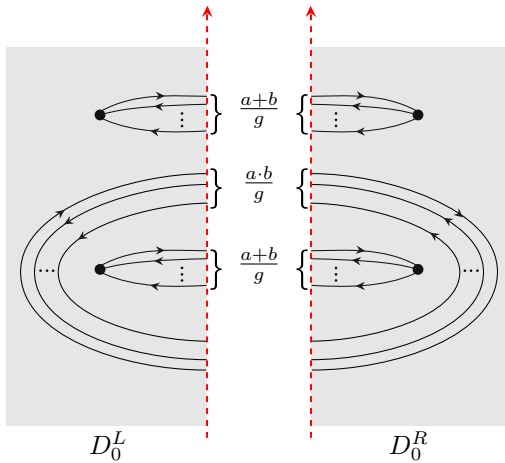
\begin{figure}[h]
  \centering
\begingroup%
  \makeatletter%
  \providecommand\color[2][]{%
    \errmessage{(Inkscape) Color is used for the text in Inkscape, but the package 'color.sty' is not loaded}%
    \renewcommand\color[2][]{}%
  }%
  \providecommand\transparent[1]{%
    \errmessage{(Inkscape) Transparency is used (non-zero) for the text in Inkscape, but the package 'transparent.sty' is not loaded}%
    \renewcommand\transparent[1]{}%
  }%
  \providecommand\rotatebox[2]{#2}%
  \newcommand*\fsize{\dimexpr\f@size pt\relax}%
  \newcommand*\lineheight[1]{\fontsize{\fsize}{#1\fsize}\selectfont}%
  \ifx\svgwidth\undefined%
    \setlength{\unitlength}{190.03343717bp}%
    \ifx\svgscale\undefined%
      \relax%
    \else%
      \setlength{\unitlength}{\unitlength * \real{\svgscale}}%
    \fi%
  \else%
    \setlength{\unitlength}{\svgwidth}%
  \fi%
  \global\let\svgwidth\undefined%
  \global\let\svgscale\undefined%
  \makeatother%
  \begin{picture}(1,0.90234814)%
    \lineheight{1}%
    \setlength\tabcolsep{0pt}%
    \put(0,0){\includegraphics[width=\unitlength,page=1]{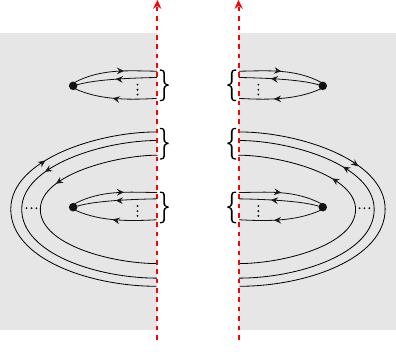}}%
    \put(0.499995,0.68489877){\color[rgb]{0,0,0}\makebox(0,0)[t]{\lineheight{1.25}\smash{\begin{tabular}[t]{c}$\frac{a+b}{g}$\end{tabular}}}}%
    \put(0.20109456,0.01352212){\color[rgb]{0,0,0}\makebox(0,0)[t]{\lineheight{1.25}\smash{\begin{tabular}[t]{c}$D^L_0$\end{tabular}}}}%
    \put(0.7972223,0.01352212){\color[rgb]{0,0,0}\makebox(0,0)[t]{\lineheight{1.25}\smash{\begin{tabular}[t]{c}$D^R_0$\end{tabular}}}}%
    \put(0.499995,0.3766118){\color[rgb]{0,0,0}\makebox(0,0)[t]{\lineheight{1.25}\smash{\begin{tabular}[t]{c}$\frac{a+b}{g}$\end{tabular}}}}%
    \put(0.499995,0.53766754){\color[rgb]{0,0,0}\makebox(0,0)[t]{\lineheight{1.25}\smash{\begin{tabular}[t]{c}$\frac{a\cdot b}{g}$\end{tabular}}}}%
  \end{picture}%
\endgroup%

  \caption{Induced orientations on $\partial B_0$.}
  \label{fig:final}
\end{figure}

This completes the construction and proves that the resulting surface remains orientable throughout the process.

\end{proof}

\begin{theorem}\label{thm:2even-boundaries}
    Given the pretzel knot $K = P(-(n+1),n+1, 2n^2q-n+1 )$ has an essential orientable surface with $2n$ boundary components and slope denominator $q$ where $n$ is an even positive integer, and $q$ is an arbitrary positive integer.
\end{theorem}
\begin{proof}

We apply Theorem \ref{thm:surface_on_-ppq} using $a =n$ and $b = 2n^2q-n$.
Observe first that $gcd(a,b) = n$ and $gcd(a+b, 2a^2) =  gcd(2n^2q, 2n^2)=2n^2$.
So, the number of boundary components is the quotient $gcd(a+b, 2a^2)/gcd(a,b) = 2n^2/n = 2n$, as claimed.
And the denominator is the quotient $(a+b)/gcd(a+b, 2a^2) = 2n^2q/2n^2 =  q$, which concludes the proof of our statement.
\end{proof}


The following result can replace the previous theorem: the pretzels are genus one knots.
\begin{theorem}\label{thm:2odd-boundaries-simplier}
    Given the pretzel knot, $K = P(-(n+1),n+2, n(2nq+1)+2 )$ has an essential orientable surface with $2n$ boundary components and slope's denominator $q$ where $n$ is an odd positive integer, and $q$ is an arbitrary positive integer.
\end{theorem}
\begin{proof}
    We employ a construction analogous to that in Theorem \ref{thm:surface_on_-ppq}, applying the algorithm of Hatcher and Oertel \cite{HO}.
    Consider the pretzel knot $K = P(p_1, p_2, p_3)$ with parameters $p_1 = -(n+1)$, $p_2 = n+2$, and $p_3 = n(2nq+1)+2$.
    We construct a candidate surface using a system of edgepaths on the diagram $\mathcal{D}$ for each tangle.
    
    Similar to the proof of Theorem \ref{thm:surface_on_-ppq}, we set up the edgepath system equations (E1)--(E4). 
    Let $x_i, y_i$ be the coordinates determining the edgepaths in the respective tangles. 
    By solving the consistency equations for the number of sheets and the vertical coordinates (Condition E3), we obtain a solution space generated by a minimal integer solution.
    
    Substituting the specific values of $p_1, p_2, p_3$ given in the theorem statement into the formula for the number of boundary components $|\partial F|$ and the boundary slope derived from the total twist $\tau$, we verify the following:
    \begin{enumerate}
        \item The number of boundary components $|\partial F|$ simplifies to $2n$.
        \item The denominator of the boundary slope simplifies to $q$.
    \end{enumerate}
    Furthermore, since the cycle of $r$-values for the final edges avoids the exceptional cases in Corollary \ref{cor:HO-2.5}, the surface is incompressible.
    The orientability is ensured by the parity of the parameters, as $n$ is odd, which complements the even case in Theorem \ref{thm:2even-boundaries}.
    Thus, the triple $(g, 2n, q)$ is realized for some genus $g$.
\end{proof}

\section{Concluding Remarks and Future Directions}

In this paper, we have established a fundamental existence theorem for the realization problem of essential surfaces in the $3$-sphere. Specifically, we proved that for any even integer $b \ge 2$ and any integer $q \ge 1$, there exists a knot $K \subset S^3$ such that the triple $(g, b, q)$ is realized by an essential surface for some genus $g \ge 0$. Combined with the parity constraints established in Propositions \ref{prop:odd} and \ref{prop:non-separating}, which force $q=1$ when $b$ is odd, these results provide a comprehensive qualitative picture of how the parity of the number of boundary components constrains the possible boundary slopes.

However, our constructive approach using pretzel knots leaves several quantitative and qualitative questions open. We conclude by proposing the following directions for future research.

\subsection{Minimality of Genus}
While Theorem \ref{main} asserts the existence of \textit{some} genus $g$, the pretzel knots utilized in our construction may not necessarily yield the minimal possible genus for a given $(b, q)$. Determining the lower bound of the genus required for realization remains a natural next step.

\begin{problem}For a given even integer $b \ge 2$ and an integer $q \ge 1$, determine the minimum genus $g_{\min}(b, q)$ such that the triple $(g_{\min}(b, q), b, q)$ is realizable in $S^3$.\end{problem}

\subsection{Universal Bounds on Euler Characteristic}
Quantitative inequalities relating the Euler characteristic $\chi(F)$ to $b$ and $q$ have been established for specific classes, such as alternating knots and Montesinos knots. It is of significant interest to determine whether these class-specific bounds can be generalized to a universal function.

\begin{problem}
Does there exist a universal function $f(b, q)$ such that the inequality $-\chi(F) \ge f(b, q)$ holds for any essential surface $F$ with $b$ boundary components and boundary slope denominator $q$ in a knot exterior? 
\end{problem}


\printbibliography

\end{document}